\documentclass[11pt,a4paper,english]{amsart}
\usepackage[utf8]{inputenc}
\usepackage[english]{babel} 

\usepackage{amsmath} 
\usepackage{amsthm} 
\usepackage{graphicx} 
\usepackage{xcolor} 
\usepackage{amsfonts}
\usepackage[biblabel]{cite}
\usepackage{bm} 
\usepackage[hidelinks]{hyperref} 
\usepackage{amssymb} 
\usepackage{tikz-cd} 
\usepackage{amscd}
\usepackage{etoolbox}
\patchcmd{\thmhead}{(#3)}{#3}{}{}

\DeclareMathOperator{\depth}{depth} 
\DeclareMathOperator{\height}{ht} 
\DeclareMathOperator{\ev}{ev} 

\newcommand{\F}{{\mathbb{F}}}
\newcommand{\fq}{\mathbb{F}_q}
\newcommand{\PP}{{\mathbb{P}}}
\newcommand{\PM}{{\mathbb{P}^{m-1}}}
\newcommand{\X}{{\mathbb{X}}}

\usepackage{mathtools} 
\DeclarePairedDelimiter\abs{\lvert}{\rvert}%
\DeclarePairedDelimiter\norm{\lVert}{\rVert}%

\makeatletter
\let\oldabs\abs
\def\abs{\@ifstar{\oldabs}{\oldabs*}}
\let\oldnorm\norm
\def\norm{\@ifstar{\oldnorm}{\oldnorm*}}
\makeatother

\newtheorem{thm}{Theorem}[section]
\newtheorem{prop}[thm]{Proposition}
\newtheorem{cor}[thm]{Corollary}
\newtheorem{lem}[thm]{Lemma}
\theoremstyle{definition}

\newtheorem{rem}[thm]{Remark} 
\newtheorem{eje}[thm]{Example} 
\newtheorem{ex}[thm]{Example}

\title[Saturation and vanishing ideals]{Saturation and vanishing ideals}
\author{Philippe Gimenez, Diego Ruano and Rodrigo San-José}
\curraddr{
\texttt{Philippe Gimenez, Diego Ruano, Rodrigo San-José:} IMUVA-Mathematics Research Institute, Universidad de Valladolid, 47011 Valladolid (Spain).
}
\email{pgimenez@uva.es;  diego.ruano@uva.es; rodrigo.san-jose@uva.es}
\date{}
\thanks{This work was supported in part by the following grants:
PID2019-104844GB-I00 funded by MCIN/AEI/10.13039/501100011033,
PGC2018-096446-B-C21 funded by MCIN/AEI/10.13039/501100011033 and ``ERDF A way of making Europe'', RYC-2016-20208 funded by MCIN/AEI/10.13039/501100011033 and ``ESF Investing in your future'', and FPU20/01311 funded by the Spanish Ministry of Universities.}

\subjclass[2020]{Primary: 13P25. Secondary: 13M10, 94B27}

\keywords{Projective codes, evaluation codes, vanishing ideal, saturation, radical}

\begin{document}

\begin{abstract}
We consider an homogeneous ideal $I$ in the polynomial ring $S=K[x_1,\dots,$ $x_m]$ over a finite field $K=\fq$ and the finite set of projective rational points $\X$ that it defines in the projective space $\PP^{m-1}$. We concern ourselves with the problem of computing the vanishing ideal $I(\X)$. This is usually done by adding the equations of the projective space $I(\PP^{m-1})$ to $I$ and computing the radical. We give an alternative and more efficient way using the saturation with respect to the homogeneous maximal ideal.
\end{abstract}

\maketitle

\begin{flushright}
{\em Dedicated to Rafael Villarreal on his 70th Birthday.}
\end{flushright}

\null\vskip 15pt

\section{Introduction}\label{se:uno}

The aim of this paper is to compute the vanishing ideal of a finite set of points in the projective space. The motivation comes from Coding Theory, in which some projective codes are defined using these type of ideals. In the affine case, the computation of the vanishing ideal of a finite set of points is straightforward, but the projective case poses some additional problems. It is known that the vanishing ideal can be obtained computing the radical of a certain ideal, and we show that it can also be obtained computing the saturation with respect to the homogeneous maximal ideal, which is more efficient.

Let $K=\mathbb{F}_q$ be a finite field, and let $S=K[x_1,\dots,x_m]$ be the polynomial ring with standard grading. Let $I\subset S$ be an ideal. We denote by $X=V_{\fq}(I)=\{P_1,\dots,P_n\}\subset \mathbb{A}^m$ the finite set of rational points in which all the polynomials of $I$ vanish. Then we can consider the vanishing ideal of $X$, $I(X)$. With this notation we define the following evaluation map:
$$
\ev_X:S/I(X) \rightarrow \mathbb{F}_q^n ,\:\: f+I(X)\mapsto \left(f(P_1),\dots,f(P_n)\right).
$$

By the definition of $I(X)$, this evaluation map is an isomorphism of $\mathbb{F}_q$-vector spaces. If we consider $L$ a vector subspace of $S/I(X)$, we can define the \textit{affine variety code} $C(I,L)$ as the image of $L$ under the evaluation map $\ev_X$. That is:
$$
C(I,L)=\ev_X(L)=\{ \ev_X(f+I(X))\mid f+I(X)\in L \}.
$$

In the light of this definition one may wonder how to compute the ideal $I(X)$. In this affine setting, the answer is quite straightforward. The ideal $I_q=I+\langle x_1^q-x_1,\dots,x_m^q-x_m \rangle$ satisfies
$$
V_{\overline{\mathbb{F}_q}}(I_q)=V_{\mathbb{F}_q}(I_q)=V_{\mathbb{F}_q}(I)=V_{\mathbb{F}_q}(I(X))=X.
$$

By Seidenberg's Lemma \cite[Prop. 3.7.15]{kreuzer1}, $I_q$ is radical. Hence, in this case $I_q=I(X)$ and we obtain the vanishing ideal directly.

Following a similar idea, one can consider evaluation codes over the projective space $\mathbb{P}^{m-1}$.  Let $I\subset S$ be an homogeneous ideal. Again, we consider $\mathbb{X}=V_{\PP^{m-1}}(I)=\{[P_1],\dots,$ $[P_n]\}\subset \mathbb{P}^{m-1}$ the finite set of projective points defined by $I$ with representatives $P_i$. Denoting the vanishing ideal of $\X$ by $I(\mathbb{X})$, we can define the following K-linear map for each degree $d$:
$$
\ev_d:S_d \rightarrow K^{n},\:\: f\mapsto \left(\frac{f(P_1)}{f_1(P_1)},\dots,\frac{f(P_n)}{f_n(P_n)}\right),
$$
where $f_i\in S_d$ are fixed homogeneous polynomials verifying $f_i(P_i)\neq 0$. Then the image of $S_d$ under $\ev_d$, denoted by $C_\mathbb{X}(d)$, is called a \textit{projective Reed-Muller type code} of degree $d$ on $\mathbb{X}$. By definition, $I(\mathbb{X})_d=\ker \ev_d$. Thus, $S_d/I(\X)_d\cong C_\mathbb{X}(d)$. It can easily be checked that the basic parameters of the code (length, dimension and minimum distance) do not depend on the choice of the polynomials $f_i$. These codes have been studied in various contexts \cite{carvalho,nestedcartesian,sorensen,geramita,gonzalez}.

In order to compute $I(\X)$, as in the affine case, a natural idea would be to add the equations of the projective space to the ideal $I$, and check whether the resulting ideal is radical. These equations correspond to the generators of the vanishing ideal of the set of all points in $\mathbb{P}^{m-1}$ \hspace{1sp}\cite{mercier}:
$$
I(\mathbb{P}^{m-1})=\langle \{x_i^qx_j-x_ix_j^q, 1\leq i <j\leq m \} \rangle.
$$

We can define $I_q=I+I(\mathbb{P}^{m-1})$ and, as before, if this ideal were radical, then it would be equal to $I(\mathbb{X})$. However, this ideal is not radical in general. In fact, we have observed that this ideal is radical only in very specific cases. In general, computing the radical may be computationally intensive. Thus, it is an interesting problem to find an easier way to compute $I(\X)$.

In Theorem \ref{sat}, we prove that we can compute the vanishing ideal $I(\X)$ using the saturation with respect to the homogeneous maximal ideal:
$$
I(\X)=(I+I(\PP^{m-1})):\mathfrak{m}^\infty).
$$

We then ask ourselves if there are many cases in which there is no need to use the saturation, i.e., $I+I(\PP^{m-1})=I(\X)$. The answer is that this rarely happens, because it is equivalent to the question of whether $I_q$ is radical or not. Following this direction, in Proposition \ref{impos}, we show that there are finite sets of points $\X\subset \PP^{m-1}$ such that there is no ideal $I\subset S$, besides $I(\X)$, such that $I+I(\PP^{m-1})=I(\X)$. 

\section{Main result}
Before providing the main result, we recall some well known results. The first one is often referred as \textit{additivity of the degree}.
\begin{prop}[\hspace{1sp}{\cite[Lem. 5.3.11]{greuel}}]\label{aditividadgrado}
Let $I\subset S$ be an homogeneous ideal and let $I=\mathfrak{q}_1\cap\cdots\cap \mathfrak{q}_m$ be its irredundant primary decomposition. Then
$$
\deg(S/I)=\sum_{\height(\mathfrak{q}_i)=\height(I)}\deg(S/\mathfrak{q}_i).
$$
\end{prop}

The vanishing ideal of a finite set of points satisfies many properties. We list some of them below.

\begin{lem}[\hspace{1sp}{\cite[Cor. 6.3.19]{kreuzer2}}]
Let $[\alpha]\in\PP^{m-1}$, with $\alpha=(\alpha_1,\dots,\alpha_m)$, and let $I_{[\alpha]}=I(\{[\alpha]\})$ its vanishing ideal. Then
$$
I_{[\alpha]}=\left(\{\alpha_ix_j-\alpha_jx_i\mid 0\leq i<j\leq m \} \right).
$$
\end{lem}

\begin{rem}
In the previous lemma, at least one $\alpha_k\neq 0$ for some $k$. Hence, we can express $I_{[\alpha]}$ in the following way:
$$
I_{[\alpha]}=(\{x_i-\frac{\alpha_i}{\alpha_k}x_k\mid i=1,\dots,n,i\neq k \}).
$$
\end{rem}

\begin{cor}\label{props}
The ideal $I_{[\alpha]}$ is prime, $\deg(S/I_{[\alpha]})=1$ and $\height(I_{[\alpha]})=m-1$.
\end{cor}
\begin{proof}
All properties follow from the fact that $S/I_{[\alpha]}\cong K[x_k]$ for some $k$, which is obvious from the previous remark.
\end{proof}

\begin{rem}\label{remarkprimary}
If we have a finite subset $\mathbb{X}\subset \PP^{m-1}$, then
$$
I(\X)=\bigcap_{[\beta]\in\X}I_{[\beta]}.
$$

Taking into account that each $I_{[\beta]}$ is prime, the previous expression is an irredundant primary decomposition of $I(\X)$.
\end{rem}

\begin{cor}\label{corgrado}
Let $\mathbb{X}\subset \PP^{m-1}$ be a finite subset. Then $\deg(S/I(\X))=\abs{\X}$, $\height(I(\X))=m-1$, and $S/I(\X)$ is Cohen-Macaulay.
\end{cor}
\begin{proof}
The first property follows from Proposition \ref{aditividadgrado} and the previous remark. The second one follows from Corollary \ref{props} and the previous remark. The last one is deduced from the fact that $\depth(I(\X))=0$ if and only if $I(\X)$ has an $\mathfrak{m}$-primary component, which is not the case because of the previous remark, and the fact that $\dim S/I(\X)=1$.
\end{proof}

The following lemma is interesting  because it relates the number of common zeros of a set of polynomials to the degree of a certain ideal, which gives a relation between Coding Theory and Commutative Algebra.

\begin{lem}[\hspace{1sp}{\cite[Lem. 3.4]{gonzalez}}]\label{lemacontarceros2}
Let $\mathbb{X}$ be a finite subset of $\mathbb{P}^{m-1}$ over a field $K$, and let $I(\mathbb{X})\subset S$ be its vanishing ideal. If $F=\{f_1,\dots,f_r\}$ is a set of homogeneous polynomials of $S\setminus \{0\}$, then the number of points of $V_\mathbb{X}(F)$ (common zeroes of $F$ which are in $\mathbb{X}$) is given by
$$
\abs{V_\mathbb{X}(F)}=\begin{cases}
\deg(S/(I(\mathbb{X}),F)) &\text{ if } (I(\mathbb{X}):(F))\neq I(\mathbb{X}), \\
0 &\text{ if } (I(\mathbb{X}):(F))= I(\mathbb{X}).
\end{cases}
$$
\end{lem}

\begin{lem}[\hspace{1sp}{\cite[Lem. 8]{engheta}}]\label{unmixgrado}
Let $I\subset J\subset S$ be unmixed homogeneous ideals with the same height. If $\deg(S/I)=\deg(S/J)$, then $I=J$.
\end{lem}

The computation of the vanishing ideal only makes sense when $\X=V_{\PP^{m-1}}(I)\neq \emptyset$. One can get $\X=\emptyset$ in several ways, for example, if $I$ is 0-dimensional, or if it has positive dimension but no common zero of the homogeneous polynomials in $I$ is in $\PP^{m-1}$ for the corresponding field $\fq$. The following lemma gives an algebraic characterization of this property.

\begin{lem}
Let $I\subset S$ be an homogeneous ideal. Then $\X=V_{\PP^{m-1}}(I)= \emptyset$ if and only if $(I(\mathbb{P}^{m-1}):I)= I(\mathbb{P}^{m-1})$.
\end{lem}
\begin{proof}
We have $\X=V_{\PP^{m-1}}(I)= \emptyset$ if and only if $I\not \subset I_{[P]}$ for any $[P]\in \PP^{m-1}$. We also have that $(I_{[P]}:I)=I_{[P]}$ if and only if $I\not\subset I_{[P]}$ because $I_{[P]}$ is prime. Therefore, $I\not \subset I_{[P]}$ for any $[P]\in \PP^{m-1}$ if and only if $(I(\mathbb{P}^{m-1}):I)= I(\mathbb{P}^{m-1})$ because
$$
(I(\mathbb{P}^{m-1}):I)=\bigcap_{[P]\in \PP^{m-1}} (I_{[P]}:I),
$$
so the result is proved.
\end{proof}

The natural way of computing $I(\X)$ from the point of view of Coding Theory is by taking the radical of $I_q=I+I(\PP^{m-1})$, similarly to what is done in the affine case (although in that case, $I_q$ is always radical). For this, we have to prove that
$$
I(\X)=\sqrt{I+I(\mathbb{P}^{m-1})}.
$$

This can be seen as an application of Hilbert's Nullstellensatz in the algebraic closure of $\fq$, or can be proved directly as in {\cite[Thm. 3.13]{jaramillo}}. Inspired by the proof of the latter, the following theorem shows a way to compute $I(\X)$ using the saturation with respect to the homogeneous maximal ideal. This is also a natural way to compute $I(\X)$ from the point of view of Commutative Algebra, because we are getting rid of the 0-dimensional components, which are meaningless in this projective setting. 
Note that saturation with respect to a specific element has been used for similar purposes in \cite[Cor. 4.4]{renteria} for certain projective binomial varieties.

\begin{thm}\label{sat}
Let $I$ be an homogeneous ideal such that $(I(\mathbb{P}^{m-1}):I)\neq I(\mathbb{P}^{m-1})$. Let $\mathbb{X}=V_{\mathbb{P}^{m-1}}(I)$ and $\mathfrak{m}=(x_1,\dots,x_m)$ the homogeneous maximal ideal. Then
$$
I(\X)=(I+I(\mathbb{P}^{m-1})):\mathfrak{m}^\infty.
$$
\end{thm}
\begin{proof}
Again we denote $I_q=I+I(\mathbb{P}^{m-1})$ and we are going to prove first that $\deg(S/I_q)=\deg(S/I(\mathbb{X}))$. We can apply Lemma \ref{lemacontarceros2} to $\mathbb{X}=\mathbb{P}^{m-1}$ and a set of generators $F$ of $I$. We obtain:
$$
\abs{\mathbb{X}}=\abs{V_{\mathbb{P}^{m-1}}(I)}=\deg(S/I_q),
$$
and because $\abs{\mathbb{X}}=\deg(S/I(\mathbb{X}))$ holds by Corollary \ref{corgrado}, we get the equality $\deg(S/I_q)=\deg(S/I(\mathbb{X}))$. 

On the other hand, we have that $\sqrt{I_q}=I(\mathbb{X})$. Thus, $\dim(S/I_q)=1$ and the primary decomposition is
$$
I_q=\mathfrak{q}_1\cap \cdots\cap \mathfrak{q}_l\cap Q,
$$
where $\dim(S/\mathfrak{q}_i)=1$, $1\leq i \leq l$, and $Q$ is the whole ring $S$ if $I_q$ is equidimensional, and an $\mathfrak{m}$-primary ideal otherwise. If we consider the irredundant primary decomposition $I(\mathbb{X})=\bigcap_{[P_i]\in \X} I_{[P_i]}$, with $\abs{\X}=n$, then we get the equality
$$
I_{[P_1]}\cap \cdots \cap I_{[P_n]}=\sqrt{\mathfrak{q}_1}\cap \cdots \cap\sqrt{ \mathfrak{q}_l}\cap \sqrt{Q}=\sqrt{\mathfrak{q}_1}\cap \cdots\cap\sqrt{ \mathfrak{q}_l}.
$$

Reordering if necessary, we have that $I_{[P_i]}=\sqrt{\mathfrak{q}_i}$, $1\leq i \leq n$ and $l\geq n$. Taking into account that $\deg(S/I_q)=\deg(S/I(\mathbb{X}))$, the additivity of the degree \ref{aditividadgrado} and $\deg(S/I_{[P_i]})=1$ for all $i$, we get
$$
\abs{\mathbb{X}}=n=\sum_{i=1}^n \deg(S/I_{[P_i]})=\sum_{i=1}^l \deg(S/\mathfrak{q}_i)\geq l.
$$

As observed before, $l\geq n$, which, together with the previous inequality, gives $l=n$. Moreover, we deduce $\deg(S/\mathfrak{q}_i)=1$ for all $i$, $1\leq i\leq n$. Therefore, using that $I_{[P_i]}=\sqrt{\mathfrak{q}_i}\supset \mathfrak{q}_i$, $1\leq i\leq n$ and Lemma \ref{unmixgrado} we have that $\mathfrak{q}_i=I_{[P_i]}$, $1\leq i\leq n$. Finally, we observe that
$$
(I_q:\mathfrak{m}^\infty)=(I_{[P_1]}:\mathfrak{m}^\infty)\cap \cdots\cap  (I_{[P_n]}:\mathfrak{m}^\infty)\cap (Q:\mathfrak{m}^\infty)=I_{[P_1]}\cap\cdots\cap I_{[P_n]}=I(\mathbb{X}),
$$
and the result holds.
\end{proof}

Theorem \ref{sat} gives a more efficient way of computing the vanishing ideal $I(\X)$ than the usual way using the radical. For the computations we needed to choose between the different computer algebra systems, the main ones for Commutative Algebra are CoCoA \cite{cocoa}, Singular \cite{singular} and Macaulay2 \cite{M2}. We chose Macaulay2 for the examples on this occasion. We have used a computer with 512GB of RAM and an AMD EPYC 7F52 processor.

\begin{eje}
We consider the 3-dimensional rational normal scroll defined by the equations given by the $2\times 2$ minors of the following matrix:
$$
M=\begin{pmatrix}
x_0 & x_1 & x_2 & x_3 & x_4 & y_0 & y_1 & y_2 & y_3 & y_4 & z_0 & z_1 & z_2 & z_3 & z_4\\
x_1 & x_2 & x_3 & x_4 & x_5 & y_1 & y_2 & y_3 & y_4 & y_5 & z_1 & z_2 & z_3 & z_4 & z_5 \\
\end{pmatrix},
$$
and let $I$ be the homogeneous ideal defined by these equations. The number of rational points of this variety on $\mathbb{F}_q$ is $(q^2+q+1)(q+1)$ \cite[Cor. 2.3]{carvalho}. We first consider the case with $q=9$. In this situation, $\abs{\X}=910$, and the computation of the saturation with Macaulay2 \cite{M2} takes 3.65 seconds. However, the computation of the radical of $I_q$ takes 1108.15 seconds, which shows the big difference in efficiency between the two methods.

If we consider the case $q=11$ instead, we have $\abs{\X}=1596$. The saturation takes 5.08 seconds, and Macaulay2 \cite{M2} is not able to compute the radical of the ideal.

For this example, we have also considered Magma \cite{magma}, which seems to have a well-optimized algorithm to compute the radical over fields of positive characteristic. Although the efficiency gap is reduced, the saturation is still more efficient than computing the radical.
\end{eje}

\begin{rem}
It is always possible to obtain the vanishing ideal using the saturation with respect to a single polynomial. Because of prime avoidance \cite[Lemma 3.3]{eisenbud} there is a homogeneous polynomial $f\in S$ such that $f\not \in I_{[P_i]}$, for every $[P_i]\in\X$, i.e., $f$ does not vanish at any of the points of $\X$. Then, following the proof of Theorem \ref{sat}, we get
$$
(I_q:f^\infty)=\left( \bigcap_{[P_i]\in\X}(I_{[P_i]}:f^\infty)\right)\cap (Q:f^\infty)=\bigcap_{[P_i]\in\X}I_{[P_i]}=I(\X).
$$
The problem is that finding such a polynomial $f$ may not be easy. However, in some specific examples, such as the following one, it can be done.
\end{rem}

\begin{ex}
Let $I$ be the homogeneous ideal of the rational normal curve defined by the equations given by the $2\times 2$ minors of the matrix
$$
N=\begin{pmatrix}
x_0 & x_1 & x_2 & x_3 & x_4 \\
x_1 & x_2 & x_3 & x_4 & x_5 \\
\end{pmatrix}.
$$
We work over the field $\F_9$, and we consider the polynomial $f=x_0-x_4-x_5$. If we define $I_9=I+I(\PP^5)$, then it is easy to check with Macaulay2 \cite{M2} that $(I_9:f^\infty)=(I_9:\mathfrak{m}^\infty)=I(\X)$, and that $f$ does not vanish at any of the points in $\X$, i.e., $(I(\X):f)=I(\X)$.
\end{ex}

Having seen how to compute the vanishing ideal $I(\X)$, one may wonder if there are many cases in which $I_q$ is saturated. If that were the case, we would not need to compute the saturation and we would get the vanishing ideal directly. An equivalent question would be to ask when the equality $I+I(\PP^{m-1})=I(\X)$ holds. It is easy to see that if one takes $I=I(\mathbb{X})$, the vanishing ideal of a finite set of points $\X\subset \mathbb{P}^{m-1}$, then $I(\X)+I(\PP^{m-1})=I(\X)$. Another trivial example would be to take an ideal $I$ with $V_{\PP^{m-1}}(I)=\PP^{m-1}$. We can also find some nontrivial examples, like the following one.

\begin{eje}\label{ejecartesiano}
Let $K=\F_4$ and $\mathbb{X}=[\mathbb{F}_2\times \mathbb{F}_2\times \mathbb{F}_4]\subset \mathbb{P}^{2}$. We can compute the vanishing ideal $I(\X)$ directly (intersecting the vanishing ideals of the points in $\X$), but we can also use \cite[Prop. 2.11]{nestedcartesian}. In any case, we obtain 
$$
I(\mathbb{X})=(x_1x_2^2+x_1^2x_2,x_1x_3^4+x_1^4x_3,x_2x_3^4+x_2^4x_3).
$$
We consider the ideal $I$, obtained by replacing the primary component $(x_1,x_2)$ of $I(\X)$ by $(x_1,x_2)^2$. Clearly $\mathbb{X}=V_{\mathbb{P}^{2}}(I)$. In this situation, it turns out that $I+I(\PP^2)=I(\X)$. This is easy to check by looking at the generators of these ideals:
$$
\begin{aligned}
I(\mathbb{P}^2)&=(x_1x_2^4+x_1^4x_2, x_1x_3^4+x_1^4x_3, x_2x_3^4+x_2^4x_3), \\
I&=(x_1x_2^2+x_1^2x_2,x_1(x_1x_3^4+x_1^4x_3),x_2(x_1x_3^4+x_1^4x_3),x_2(x_2x_3^4+
x_2^4x_3)),\\
I(\mathbb{X})&=(x_1x_2^2+x_1^2x_2,x_1x_3^4+x_1^4x_3,x_2x_3^4+x_2^4x_3). \\
\end{aligned}
$$

Similar examples can be constructed by considering $\mathbb{X}=\mathbb{F}_{p^l}\times\mathbb{F}_{p^l}\times \mathbb{F}_{p^{l'}}$, with $l\mid l'$, and increasing the multiplicity of the component $(x_1,x_2)$.
\end{eje}

Even though we can construct several nontrivial examples, one can observe that in order to do so we have not strayed away too much from $I(\PP^{m-1})$ and $I(\X)$ (we have just used an ideal $I(\X)$ that shares some generators with $I(\PP^{m-1})$ and modified it a little). In fact, in most cases we have encountered, $I_q$ was not saturated. The next result shows that there are some finite sets of points $\X$ such that there are no nontrivial homogeneous ideals $I$ with $V_{\PP^{m-1}}(I)=\X$ verifying $I+I(\PP^{m-1})=I(\X)$.

\begin{prop}\label{impos}
Let $\X\subset \PM$ be a finite set of points such that the degree of the elements of a minimal generating set of $I(\X)$ is lower than $q+1$. Then $I+I(\PM)=I(\X)$ if and only if $I=I(\X)$.
\end{prop}
\begin{proof}
Let $I$ be an homogeneous ideal verifying $I+I(\PM)=I(\X)$. Obviously, $I\subset I(\X)$, and we have to prove the other inclusion. The degree of the minimal generators of $I(\PM)$ is $q+1$. Therefore, the minimal generators of $I(\X)$, of degree lower than $q+1$, must all be in $I$, which proves the result.
\end{proof}

\begin{eje}
Let $K=\F_4$, and let $\X=[\F_2\times \F_2\times \F_2]\subset \PP^2$. The vanishing ideal $I(\X)$ is the same as $I(\PP^2)$ in $\F_2$. Therefore, we have
$$
I(\X)=(x_1^2x_2-x_1x_2^2,x_1^2x_3-x_1x_3^2,x_2^2x_3-x_2x_3^2).
$$
The generators of $I(\X)$ are of degree $3<5=q+1$. Consequently, we can use Proposition \ref{impos} to assert that there is no homogeneous ideal $I$, besides $I(\X)$, such that $I+I(\PP^2)=I(\X)$.
\end{eje}

In the proof of \ref{sat} we showed that $\deg(S/I_q)=\deg(S/I(\X))$. Also, taking into account that $\dim( S/I(\X))=1$ and that $\sqrt{I_q}=I(\X)$, we get $\height(I_q)=\height(I(\X))$. As we have said, in most cases, $I_q\neq I(\X)$. Consequently, we would have $I_q\subsetneq I(\X)$, but $\deg(S/I_q)=\deg(S/I(\X))$. The following example illustrates this fact, which seems to contradict \cite[Lem. 2.10 (b)]{geramita}.

\begin{ex}
We consider again the set $\X=[\mathbb{F}_2\times \mathbb{F}_2\times \mathbb{F}_4]\subset \mathbb{P}^{2}$ from example \ref{ejecartesiano}. We can replace the primary component $(x_1,x_3)$ by $(x_1,x_3)^2$ in the primary decomposition of $I(\X)$, which gives the following ideal:
$$
I=I(\X)\cap (x_1,x_3)^2=(x_1(x_1x_2^2+x_1^2x_2),x_1(x_1x_2x_3+x_2^2x_3),x_1x_3^4+x_1^4x_3,x_3(x_2x_3^4+x_2^4x_3)).
$$
We can define $I_4=I+I(\PP^2)$ and it is easy to check with Macaulay2 \cite{M2} that $I_4\subsetneq I(\X)$, $\height(I_4)=\height(I(\X))=2$ and $\deg(S/I_4)=\deg(S/I(\X))$, which contradicts \cite[Lem. 2.10 (b)]{geramita}. Increasing the multiplicity of any primary component of $I(\X)$, besides $(x_1,x_2)$, we get more examples of ideals $I$ such that $I_4=I+I(\PP^2)$ is not saturated and has the same degree and height as $I(\X)$. Note that this does not contradict Lemma \ref{unmixgrado} since $I_4$ is not unmixed. 
\end{ex}

\bibliographystyle{abbrv}

\end{document}